\documentclass{amsart}

\usepackage{amsmath,amssymb,amsthm}\allowdisplaybreaks[4]
\usepackage{url}      		
\usepackage{graphicx}	
\usepackage{verbatim}       
\usepackage{layout}
\usepackage{galois}

\newtheorem{theorem}{Theorem}
\newtheorem{lemma}{Lemma}
\newtheorem{proposition}{Proposition}
\newtheorem{remark}{Remark}

\newtheorem{example}{Example}

\graphicspath{%
    {converted_graphics/}
    {/}
}
\begin{document}


\title[W-like maps with various instabilities]{$W$-like maps with various instabilities of acim's }
\thanks{The research of the author is supported by NSERC grants. }
\author{Zhenyang Li}
\address{Department of Mathematics and Statistics, Concordia University,
1455 de Maisonneuve Blvd. West, Montreal, Quebec H3G 1M8, Canada}
\email{zhenyangemail@gmail.com}

\date{\today }

\begin{abstract}
This paper generalizes the results of \cite{Li} and then provides an interesting example. We construct a family  of $W$-like maps $\{W_a\}$ with a turning fixed point having slope $s_1$ on one side and $-s_2$ on the other. Each $W_a$ has an absolutely continuous invariant measure $\mu_a$.  Depending on whether $\frac{1}{s_1}+\frac{1}{s_2}$ is larger, equal or smaller than 1, we show that the limit of $\mu_a$ is a singular measure, a combination of singular and absolutely continuous measure  or an absolutely continuous measure, respectively. It is known that the invariant density of a single piecewise expanding map  has a positive lower bound on its support. In Section \ref{S:example} we give an example showing that in general,  for a family of piecewise expanding maps with slopes larger than 2 in modulus and converging to a piecewise expanding map,  their invariant densities do not necessarily have a positive lower bound on the support.
\end{abstract}

\maketitle

\pagestyle{myheadings}
\markboth{\textsc{ Families of piecewise expanding maps}}
{\textsc{Zhenyang Li}}

\section{Introduction}

In practice, due to external noise, or roundoff errors in computation, there is a natural interest in the stability of properties of chaotic dynamical systems under small perturbations. If we consider a family of piecewise expanding maps $\tau_a:I\rightarrow I$, $a > 0$ with absolutely continuous invariant measures (acim's) $\mu_a$, converging to a piecewise expanding map $\tau_0$ with acim $\mu_0$, then
under general assumptions $\mu_a$'s converge to $\mu_0$. One such assumption is that $\inf|\tau'_a|>2$ for all $a > 0$ (see \cite{BS}, \cite{G2}, \cite{G3}
or \cite{KL99}). This is useful in the study of the metastable systems \cite{THW}, or to approximate the invariant densities \cite{G4}.

Keller \cite{K} introduced the family  of $\{W_a\}$ maps that are piecewise expanding, ergodic transformations with a ``stochastic singularity", i.e., $\mu_a$'s converge to a singular measure. This occurs because of the existence of diminishing invariant neighborhoods of the turning fixed  point. The slopes of the Keller's $W_a$ maps converge to 2 and -2 on the left and right hand sides of the turning fixed point, respectively.

Given two numbers, $s_1$ and $s_2$, greater than 1, we consider a $W$-like map with one turning fixed  point having slope $s_1$ on one side and $-s_2$ on the other.
In \cite{Li}, the authors considered the special case where $s_1=s_2=2$. Their perturbed maps $W_a$ are piecewise expanding with slopes strictly greater than 2 in modulus and are exact with their acim's supported on all of [0, 1]. The standard bounded variation method \cite{BG} cannot be applied in this setting as the
slopes of the maps in that family are not uniformly bounded away from 2. Other methods, for example,  those studied in \cite{DFS}, \cite{Ko} and \cite{M} cannot be  applied either. Using the main result of \cite{G1},  it can be shown that the $\mu_a$'s converge to $\frac2 3\mu_0+\frac1 3\delta_{(\frac1 2)}$, where $\delta_{(\frac1 2)}$ is the Dirac measure at point $1/2$ and $\mu_0$ is the acim of the $W_0$ map. Thus, the family of measures $\mu_a$ approach a combination of an absolutely continuous and a singular measure rather than the acim of the limit map. Similar instability was also shown in \cite{EM} for a countable family of transitive Markov maps approaching Keller's $W_0$ map.

In this paper, we construct a family of maps for which the instability of the acim's has a global character, not a local one.
In the more general case considered in this paper, with $s_1$, $s_2$ not necessarily equal to 2, we will discuss the limits of the acim's $\mu_a$ of the $\{W_a\}$ maps. We have three cases:\\
(I) If $\frac{1}{s_1}+\frac{1}{s_2}>1$, then $\mu_a$'s converge $*$-weakly to $\delta_{(\frac 12)}$.\\
(II) If $\frac{1}{s_1}+\frac{1}{s_2}=1$, then $\mu_a$'s converge $*$-weakly to \[\frac{(qs_1+ps_2-p-q)(s_2+2)}{(qs_1+ps_2-p-q)(s_2+2)+2rs_1s_2^2}\mu_0+\frac{2rs_1s_2^2}{(qs_1+ps_2-p-q)(s_2+2)+2rs_1s_2^2} \delta_{(\frac 12)},\]
where $p$, $q$ and $r$ are parameters defining our family of maps.\\
(III) If $\frac{1}{s_1}+\frac{1}{s_2}<1$, then $\mu_a$'s converge to $\mu_0$.\\
Additionally, in Theorem \ref{way2}, we prove that in case (III) the densities of the $\mu_a$'s are uniformly bounded. The first case of our result contains the example in which Keller \cite{K} obtained the ``stochastic singularity." In the second case, the limit measure is a combination of an absolutely continuous and a singular measure, and this combination is varying according to $p$, $q$ and $r$ for fixed $s_1$ and $s_2$. This is a generalization of the result of \cite{Li}. In the third case, we have a map with a stable acim.

At the end of the paper, we use our main results to provide an interesting example. Keller \cite{K78} and Kowalski \cite{Ko} proved that for a piecewise expanding map $\tau: I\rightarrow I$ with $\frac{1}{|\tau'(x)|}$ being a function of bounded variation, the density of the acim of $\tau$ has a uniform positive lower bound on its support. We construct a family of piecewise expanding, piecewise linear maps $\tau_n$ such that $\tau_n$ are exact on $[0,1]$, $\tau_n$ converge to $\tau=W_0$ ($s_1=s_2=2$), $|\tau_n'|>2$ for all $n$ but the densities of the acims $\mu_n$'s do not have a uniform positive lower bound.

In Section \ref{S:main}, we introduce our family of $W_a$ maps and state the main result. In Section \ref{S:proofs} we present the proofs. In Section \ref{S:example}, we show the example related to the results of Keller \cite{K78} and Kowalski \cite{Ko}.

\section{Family of $W_a$ maps and the main result}\label{S:main}
Let $s_1$, $s_2>1$ and $p$, $q$, $r>0$. We consider the family $\{W_a: 0\le a\}$ of maps of $[0,1]$ onto itself defined by
\begin{equation}
W_a(x)=\begin{cases} 1-\frac{2(s_1+pa)}{s_1-1+pa-2ra}x \ ,\ \text{\ for\ } \ 0\le x<\frac 1 2-\frac{\frac 1 2 +ra}{s_1+pa}\ ;\\
                     (s_1+pa)(x-1/2)+1/2+ra \ ,\ \text{\ for\ } \ \frac 1 2-\frac{\frac 1 2 +ra}{s_1+pa}\le x<1/2\ ;\\
                      -(s_2+qa)(x-1/2)+1/2+ra \ ,\ \text{\ for\ } \ 1/2\le x< \frac 1 2+\frac{\frac 1 2 +ra}{s_2+qa}\ ;\\
                     1+\frac{2(s_2+qa)}{s_2-1+qa-2ra}(x-1) \ ,\ \text{\ for\ } \ \frac 1 2+\frac{\frac 1 2 +ra}{s_2+qa}\le x\le 1\ .\\
\end{cases}
\end{equation}
For each choice of $s_1$, $s_2>1$, $p$, $q$, $r>0$, we consider only $a>0$ such that $0\leq W_a(x)\leq1$ for $x\in [0,1]$.

An example of a $W_a$ map is shown in Fig.\ref{fig:wequal1}. Fig.\ref{fig:wequal1}(a) is the unperturbed $W_0$ map with turning fixed point at 1/2 and $s_1=3/2$, $s_2=3$. Fig.\ref{fig:wequal1}(b) is the perturbed map $W_a$, with $a=0.05$, $r=2$, $p=3$, $q=2$. The slope of the second branch is $s_1+pa=1.65$, the slope of the third branch is $s_2+qa=3.1$, and $W_{0.05}(1/2)=1/2+ra=0.6$.
\begin{figure}[h] 
  \centering $
  \begin{array}{cc}
  \includegraphics[width=2.3in,height=2.3in,keepaspectratio]{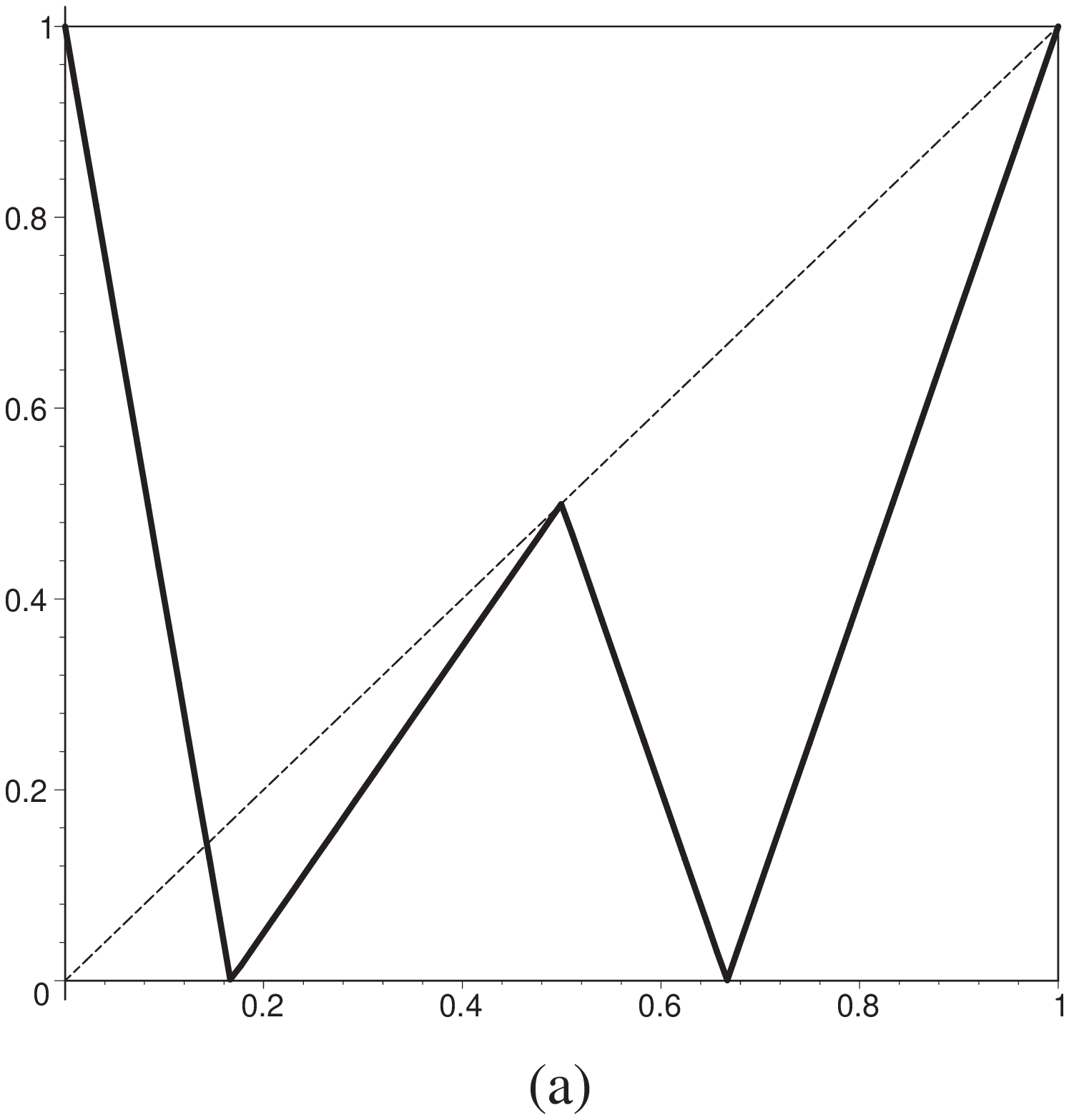} & \includegraphics[width=2.3in,height=2.3in,keepaspectratio]{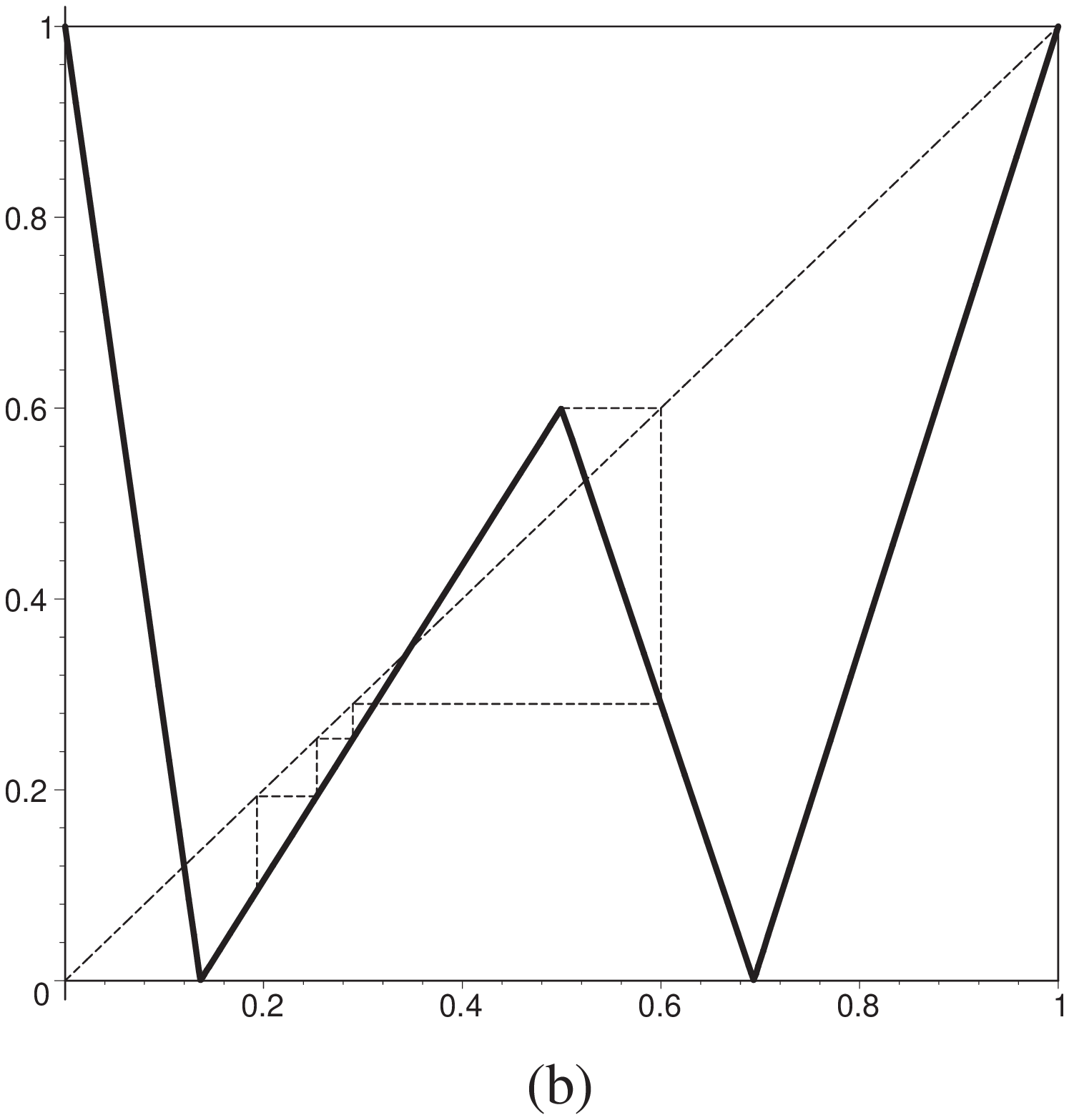}
  \end{array}$
  \caption{The $W$-like maps with $\frac{1}{s_1}+\frac{1}{s_2}=1$: (a) $W_0$ with $s_1=3/2$ and $s_2=3$, (b) $W_a$ with $s_1=3/2$, $s_2=3$; $a=0.05$; $r=2$, $p=3$, $q=2$; also several initial points of the trajectory of $1/2$.}
  \label{fig:wequal1}
\end{figure}

Every $W_a$ has a unique absolutely continuous invariant measure $\mu_a$ since all the slopes are greater than 1 in modulus. We will show later that, for $\frac{1}{s_1}+\frac{1}{s_2}\leq1$, $\mu_a$ is supported on $[0,1]$ and for $\frac{1}{s_1}+\frac{1}{s_2}>1$ it is supported on a subinterval around 1/2. $W_a$ is an exact map with the measure $\mu_a$.
Let $h_a$ denote the normalized density of $\mu_a$, $a\ge 0$. Since the $W_0$ map is a Markov one, it is easy to check that
\begin{equation}
h_0=\begin{cases} \frac{2s_1(s_2+1)}{2s_1s_2+s_1-s_2} \ ,\ \text{\ for\ } \ 0\le x<1/2\ ;\\
                      \frac{2s_2(s_1-1)}{2s_1s_2+s_1-s_2}\ ,\ \text{\ for\ } \ 1/2\le x\le 1\ .\\
\end{cases}
\end{equation}

Our main result is the following theorem
\begin{theorem}\label{main}
As $a\to 0$ the measures $\mu_a$ converge $*$-weakly to the
 measure\\
(I) $\delta_{(\frac 12)}$, if $\frac{1}{s_1}+\frac{1}{s_2}>1$;\\
(II) $\frac{(qs_1+ps_2-p-q)(s_2+2)}{(qs_1+ps_2-p-q)(s_2+2)+2rs_1s_2^2}\mu_0+\frac{2rs_1s_2^2}{(qs_1+ps_2-p-q)(s_2+2)+2rs_1s_2^2} \delta_{(\frac 12)}$, if $\frac{1}{s_1}+\frac{1}{s_2}=1$;\\
(III) $\mu_0$, if $\frac{1}{s_1}+\frac{1}{s_2}<1$,\\
where $\delta_{(\frac 12)}$ is the Dirac measure at point $1/2$.
\end{theorem}

The proof relies on the general formula for invariant densities of piecewise
linear maps \cite{G1} and direct calculations. Most objects and quantities we use depend on the parameter $a$.
We suppress $a$ from the notation to make it simpler.

In case (III), we actually prove a little more:
\begin{theorem}\label{way2}
If $\frac{1}{s_1}+\frac{1}{s_2}<1$, then the normalized invariant densities $\{h_a\}$ are uniformly bounded for given $p$, $q$ and $r$. Consequently, we obtain Theorem \ref{main}(III).
\end{theorem}

\section{Proofs} \label{S:proofs}

This section contains the proofs of Theorems \ref{main} and \ref{way2}, divided
into a number of steps.

\subsection{Assume $\frac{1}{s_1}+\frac{1}{s_2}>1$}

Let \[x^\ast_l=\frac{s_1-1+pa-2ra}{2(s_1-1+pa)}\]
and \[x^\ast_r=\frac{s_2s_1-s_2+(2rs_1-q+ps_2+qs_1)a+(2rp+pq)a^2}{2(s_1-1+pa)(s_2+qa)}.\]  $x^\ast_l$ is the fixed point on the second branch of $W_a$,  and $x^\ast_r$ is the preimage of $x^\ast_l$ under the third branch of $W_a$. Both $x^\ast_r$ and  $x^\ast_l$ converge to $\frac 12$ as $a$ approaches 0. For small $a$, we have
$$W_a(1/2)-x^\ast_r=\frac{ra\left[s_1s_2-s_1-s_2+a(qs_1+ps_2-p-q+pqa)\right]}{(s_1-1+pa)(s_2+qa)}<0.$$
In this case, we have $W_a([x^\ast_l,x^\ast_r])\subseteq [x^\ast_l,x^\ast_r]$. $W_a|_{[x^\ast_l,x^\ast_r]}$ is a skewed tent map with $W_a(1/2)>1/2$; it is known that with acim $\mu_a$, it is exact on $[x^\ast_l,W_a(1/2)]$. Since $\mu_a$ is concentrated on $[x^\ast_l,x^\ast_r]$, we conclude that $\mu_a$ converge $*$-weakly to $\delta_{(\frac 12)}$. This proves Theorem \ref{main}(I).

Fig.\ref{fig:W-map_invar_box} shows an example with $a=0.05, r=2, p=3, q=2;$ $s_1=4/3, s_2=5/2$.
\begin{figure}[h] 
  \centering
  \includegraphics[width=2.8in,height=2.8in,keepaspectratio]{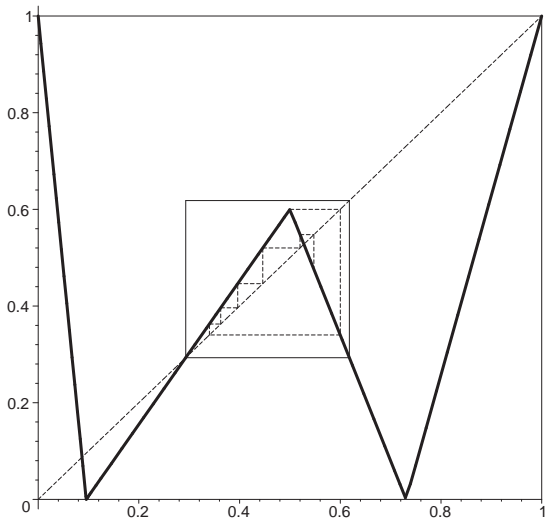}
  \caption{The $W_a$ map with $\frac{1}{s_1}+\frac{1}{s_2}>1$}
  \label{fig:W-map_invar_box}
\end{figure}

\subsection{Formula for the non-normalized invariant density of $W_a$ if $\frac{1}{s_1}+\frac{1}{s_2}\leq1$ }

An example of a map $W_a$ is shown in Fig.\ref{fig:wequal1}. We have the following proposition.
\begin{proposition}
For $\frac{1}{s_1}+\frac{1}{s_2}\leq 1$, the map $W_a$ has an absolutely continuous invariant measure $\mu_a$ supported on $[0,1]$ and the map $W_a$ with respect to $\mu_a$ is exact.
\end{proposition}
\begin{proof}
$W_a$ is a piecewise expanding transformation. From the general theory (see for example \cite{BG}), it follows that it is enough to show that the images $W_a^n(J)$ grow to cover all $[0,1]$ as $n\rightarrow \infty$, for any interval $J\subset[0,1]$. Since $W_a$ is expanding,  $W_a^n(J)$ grow until some image $W_a^{n_0}(J)$ contains an internal partition point. If this point is not 1/2, then $W_a^{n_0+2}(J)$ contains the repelling fixed point 1. Then its images grow to cover all of $[0,1]$. If this point is 1/2, we proceed as follows. First, assume that $\frac{1}{s_1}+\frac{1}{s_2}< 1$. Consider a small neighborhood $J=(z_1,z_2)$ around 1/2 with length $\ell$, then
$$\min\limits_{z_2-z_1=\ell}\max\left\{(\frac 12-z_1)(s_1+pa),(z_2-\frac 12)(s_2+qa)\right\}=\frac{1}{\frac{1}{s_1+pa}+\frac{1}{s_2+qa}}\ell>\ell.$$
Thus, the interval $J$ will grow until its image covers two partition points of $W_a$. Then the second  iteration afterward will cover $[0,1]$. Therefore, $W_a$ is exact with respect ot $\mu_a$.

Assume $\frac{1}{s_1}+\frac{1}{s_2}= 1$. If $a\neq 0$, then $\frac{1}{\frac{1}{s_1+pa}+\frac{1}{s_2+qa}}>1$, which implies $W_a$ is exact with respect to $\mu_a$.
In the case $a=0$, we first note that 1/2 is a turning fixed  point. Take again a small interval $J=(z_1,z_2)\ni 1/2$. Its image is an interval $(z,1/2)$. It will grow under iteration and its iterations still contain 1/2. It will grow until its image covers another partition point of $W_a$. Then, the second  iteration afterward will covers all of $[0,1]$. Thus, $W_a$ is again exact with respect to $\mu_a$.
\end{proof}
We adapt the general formulas of \cite{G1} to our case and obtain the following lemma:
\begin{lemma}\label{Just}
$(I)$ N=4, K=2, L=0\ ;\\
$(II)$ $\alpha=\left(1,1/2+ra,1/2+ra,1\right)$, $\beta=\left(\beta_1,\beta_2,\beta_3,\beta_4\right)$, where $\beta_1=-\frac{2(s_1+pa)}{s_1-1+pa-2ra}$, $\beta_2=s_1+pa$, $\beta_3=-(s_2+qa)$ and $\beta_4=\frac{2(s_2+qa)}{s_2-1+qa-2ra}$, $\gamma=\left(0,0,0,0\right)$\ ;\\
$(III)$ The digits $A=\left(a_1,a_2,a_3,a_4\right)$, where $a_1=-1, a_2=\frac{s_1-1+pa-2ra}{2}$, $a_3=-\frac{s_2+1+qa+2ra}{2}$, $a_4=\frac{s_2+1+qa+2ra}{s_1-1+pa-2ra}$\ ;\\
$(IV)$ There are two $c_i$'s, which are $c_1=(1/2,2)$ and $c_2=(1/2,3)$, and $j(c_1)=2$, $j(c_2)=3$. Then, $W_u=\{c_1,c_2\}$,$W_l=\emptyset$, $U_l=\{c_2\}$,$U_r=\{c_1\}$\ ;\\
$(V)$ $\beta(c_1,1)=s_1+pa$ since  $j(c_1)=2$, then $\beta(c_1,2)=-(s_1+pa)(s_2+qa)$ and $\beta(c_1,k)=-(s_2+qa)(s_1+pa)^{k-1}$ up to some $k$ which is the first moment $j$ when the $W_a^j(1/2)$ is less than $\frac 12 -\frac{1/2+ra}{s_1+pa}$, and is the same one defined in Lemma  \ref{L:estimates}\ ;\\
$(VI)$ $\beta(c_2,1)=-(s_2+qa)$ since  $j(c_2)=3$, then $\beta(c_2,2)=(s_2+qa)^2$ and $\beta(c_2,k)=(s_2+qa)^2(s_1+pa)^{k-2}$ up to the same $k$ in part $(e)$, $W_a^n(c_1)=W_a^n(c_2)$ for all $n$\ ;\\
$(VII)$ Based on $(VI)$, we have the following for the matrix $S=\left(S_{i,j}\right)_{1\leq i,j\leq 2}$\ :\\
For $c_1\in U_r$
$$S_{1,1}=\sum_{n=1}^\infty \frac{\delta(\beta((c_1,n)>0))\delta(W_a^n(c_1)>1/2)+\delta(\beta((c_1,n)<0))\delta(W_a^n(c_1)<1/2)}{|\beta(c_1,n)|},$$
$$S_{1,2}=\sum_{n=1}^\infty \frac{\delta(\beta((c_1,n)>0))\delta(W_a^n(c_1)>1/2)+\delta(\beta((c_1,n)<0))\delta(W_a^n(c_1)<1/2)}{|\beta(c_1,n)|}.$$

For $c_2\in U_l$
$$S_{2,1}=\sum_{n=1}^\infty \frac{\delta(\beta((c_2,n)<0))\delta(W_a^n(c_2)>1/2)+\delta(\beta((c_2,n)>0))\delta(W_a^n(c_2)<1/2)}{|\beta(c_2,n)|},$$
$$S_{2,2}=\sum_{n=1}^\infty \frac{\delta(\beta((c_2,n)<0))\delta(W_a^n(c_2)>1/2)+\delta(\beta((c_2,n)>0))\delta(W_a^n(c_2)<1/2)}{|\beta(c_2,n)|}.$$
\end{lemma}
\begin{remark}
It follows from $(V,VI)$ of Lemma 1 that $$S_{1,1}=S_{1,2}\ ,\ S_{2,1}=S_{2,2}\ \text{and }S_{1,1}=\frac{s_2+qa}{s_1+pa}S_{2,2}\ .$$
Let $Id$ be the $2\times2$ identity matrix and let $V=[1,1]$.  Then, for the solution, $D=[D_1,D_2]$, of the system :
$$\left(-S^T+Id\right)D^T=V^T,\eqno(1)$$
we have $D_1=D_2$. Let us denote them by $\Lambda$.
\end{remark}
Let ${I_1,I_2,I_3,I_4}$ be the partition of $I=[0,1]$ into maximal intervals of monotonicity of $W_a$: $I_1=[0,\frac{s_1-1+pa-2ra}{2(s_1+pa)}),I_2=(\frac{s_1-1+pa-2ra}{2(s_1+pa)},1/2),I_3=(1/2,\frac{s_2+1+qa+2ra}{2(s_2+qa)})$ and $I_4=(\frac{s_2+1+qa+2ra}{2(s_2+qa)},1]$. We define the following index function:
$$j(x)= j\text{ for } x\in I_j, j=1,2,3,4,$$
and
$$j(c_1)=2,j(c_2)=3.$$

We  define the cumulative slopes for iterates of points as follows:
$$\beta(x,1)=\beta_{j(x)}, \ \text{ and }\beta(x,n)=\beta(x,n-1)\cdot\beta_{j(W_a^{n-1}(x))}, \ \ \ \ n\geq2.$$

In particular, we have
$$\beta(1/2,n)= (s_1+pa)\cdot W_a'(W_a(1/2)) \cdot W_a'(W_a^2(1/2))\cdots  W_a'(W_a^{n-1}(1/2)) \ ,$$
which is the cumulative slope along the $n$ steps of trajectory of $1/2$. Recall that $k$ is the first moment $j$ when the $W_a^j(1/2)$ is less than $\frac 12 -\frac{1/2+ra}{s_1+pa}$. Let $k_1=[\frac{2}{3}k]$ (the integer part of $2k/3$). Note that $k_1\rightarrow\infty$ as $a\rightarrow 0$. Let
\begin{equation*} \chi^s(t,x)=\begin{cases} \chi_{[0,x]} \ \ & \ \text{for}\ \ t>0\ ;\\
                                            \chi_{[x,1]} \ \ & \ \text{for}\ \ t<0\ .
\end{cases}
\end{equation*}
Now, we can obtain the following formula for $f_a$:
\begin{lemma}  \label{density_fa}
Let
\begin{eqnarray*}
 f_a= 1+ (1+\frac{s_1+pa}{s_2+qa})\Lambda \left(\sum_{n=1}^\infty \frac{\chi^s(\beta(1/2,n),W_a^n(1/2))}{|\beta(1/2,n)|}
 \right).
\end{eqnarray*}
Then $f_a$ is $W_a$ invariant non-normalized density.
Furthermore, for small $a>0$, we have:\\
(I) If $\frac{1}{s_1}+\frac{1}{s_2}=1$, then $\Lambda<-1$\ ;\\
(II) If $\frac{1}{s_1}+\frac{1}{s_2}<1$, the sign of $\Lambda$ depends on $s_1$ and $s_2$, can be either positive or negative depending on the sign of  $\vartheta=1-\left(\frac{s_1+s_2}{s_1s_2}+\frac{s_1+s_2}{s_2^2(s_1-1)}\right)=
1-\frac{s_1+s_2}{s_1s_2}\left(1+\frac{s_1}{s_2(s_1-1)}\right)$. The case when $\vartheta=0$ is discussed at the end of Section \ref{S:proofs}.
\end{lemma}
\begin{proof}
By the Theorem 2 in \cite{G1}, it follows from $(IV, V, VI)$ of Lemma \ref{Just} that:
\begin{eqnarray*}
 f_a&=&1+D_1 \sum_{n=1}^\infty \frac{\chi^s(\beta(c_1,n),W_a^n(c_1))}{|\beta(c_1,n)|}+D_2 \sum_{n=1}^\infty \frac{\chi^s(-\beta(c_2,n),W_a^n(c_2))}{|\beta(c_2,n)|} \\
 & =& 1+\Lambda \sum_{n=1}^\infty \frac{\chi^s(\beta(c_1,n),W_a^n(1/2))}{|\beta(c_1,n)|}+\Lambda \sum_{n=1}^\infty \frac{\chi^s(-\beta(c_2,n),W_a^n(1/2))}{|\beta(c_2,n)|}\\
 &=& 1+ (1+\frac{s_1+pa}{s_2+qa})\Lambda \left(\sum_{n=1}^\infty \frac{\chi^s(\beta(1/2,n),W_a^n(1/2))}{|\beta(1/2,n)|}
 \right).
\end{eqnarray*}
Since
\begin{eqnarray*}   S_{1,1}&\geq&\frac{1}{s_1+pa}+\frac{1}{s_2+qa}\sum\limits_{n=1}^{k_1-1}\frac{1}{(s_1+pa)^n}=\frac{1}{s_1+pa}+\frac{1}{s_2+qa}\frac{1-\frac{1}{(s_1+pa)^{k_1-1}}}{s_1+pa-1}\ ,\\
S_{1,1}&\leq&\frac{1}{s_1+pa}+\frac{1}{s_2+qa}\sum\limits_{n=1}^{\infty}\frac{1}{(s_1+pa)^n}=\frac{1}{s_1+pa}+\frac{1}{s_2+qa}\frac{1}{s_1+pa-1}\ , \end{eqnarray*}
and $\Lambda=\frac{1}{1-\frac{s_1+s_2+pa+qa}{s_2+qa}S_{1,1}}$, we have
\begin{equation}\label{estimate_A}
\Lambda_l=\frac{1}{1-(\kappa+\eta(1-\frac{1}{(s_1+pa)^{k_1-1}}))}\leq \Lambda\leq \frac{1}{1-(\kappa+\eta)}=\Lambda_h\ ,
\end{equation}
where $\kappa=\frac{s_1+s_2+pa+qa}{(s_1+pa)(s_2+qa)}$ , $\eta=\frac{s_1+s_2+pa+qa}{(s_2+qa)^2(s_1+pa-1)}$.

To obtain the upper bound of $S_{1,1}$, we assume $s_1<s_2$. For $s_1>s_2$ the calculations differ slightly.

(I) Note that for small $a$ both estimates $\Lambda_l$ and $\Lambda_h$ are smaller than $-1$ since both $\kappa$ and $\eta$ are smaller than 1 and close to 1. Furthermore, as $a$ approaches 0, both $\kappa$ and $\eta$ approach 1.

(II) As $a$ approaches 0, $\kappa$ and $\eta$ approach $\frac{s_1+s_2}{s_1s_2}$ and $\frac{s_1+s_2}{s_2^2(s_1-1)}$, respectively. Again, note that for small $a$, estimates $\Lambda_l$ and $\Lambda_h$ can be either positive or negative, and they have the same sign.
\end{proof}

For small positive $a$, the first image of $1/2$ is $W_a(1/2)=1/2+ra$ and the next one falls just below the fixed point $x^\ast_l$ slightly less than $1/2$.
The following images form a decreasing sequence until they  go below $\frac 12 -\frac{1/2+ra}{s_1+pa}$.
Since $k$ is the first iteration $j$ when the $W_a^j(1/2)$ is less than $\frac 12 -\frac{1/2+ra}{s_1+pa}$, the consecutive cumulative
slopes of $1/2$ are
$$ (s_1+pa),-(s_1+pa)(s_2+qa),-(s_1+pa)^2(s_2+qa),\dots, -(s_1+pa)^{k-1}(s_2+qa)\ ,$$
and
\begin{equation}\label{simple_fa}
 f_a= 1+ (1+\frac{s_1+pa}{s_2+qa})\Lambda \left(\frac{\chi_{[0,W_a(1/2)]}}{(s_1+pa)}+\sum _{j=2}^k\frac{\chi_{[W_a^j(1/2),1]}}{(s_1+pa)^{j-1}(s_2+qa)}+\dots
 \right).
\end{equation}

\subsection{Estimates, normalizations and integrals on $f_a$ for $\frac{1}{s_1}+\frac{1}{s_2}\leq 1$}

Remembering that $k=\min\{j\geq 1 :  W_a^j(1/2)\leq\frac 12 -\frac{1/2+ra}{s_1+pa}\}$ and $k_1=[\frac{2}{3}k]$ (the integer part of $2k/3$), we will give the estimates on $f_a$.

Let us define
$$g_l=\frac{\chi_{[0,W_a(1/2)]}}{s_1+pa}+\frac{1}{s_2+qa}\sum _{j=2}^{k_1}\frac{\chi_{[W_a^j(1/2),1]}}{(s_1+pa)^{j-1}}\ ,$$
and
$$g_h=g_l+\frac{1}{s_2+qa}\sum_{j=0}^\infty \frac 1{(s_1+pa)^{j+k_1}}=g_l+\frac 1{(s_2+qa)(s_1+pa-1)(s_1+pa)^{k_1-1}}\ .$$
Also, let $\chi_1=\chi_{[0,1/2+ra]}$, $\chi_j=\chi_{[W_a^j(1/2),1/2+ra]},
 j=2,3,\ldots,k_1$, $\chi_{c}=\chi_{(1/2+ra,1]}$.
\subsubsection{Estimates on $f_a$ if $\frac{1}{s_1}+\frac{1}{s_2}=1$  }\label{equals1}

We have the following lemma:
\begin{lemma}\label{L:estimates1}
For the family of $W_a$ maps, if $\frac{1}{s_1}+\frac{1}{s_2}=1$, we have\\
$(I)$ $W_a(1/2)=1/2+ra$, $W_a^2(1/2)=-ra(s_2+qa)+1/2+ra$, and for $3\leq m\leq k$, we have $W_a^m(1/2)=-a^2(s_1+pa)^{m-2}\frac{r(qs_1+ps_2-p-q)+rpqa}{s_1+pa-1}+\frac{s_1-1+pa-2ra}{2(s_1+pa-1)}$;\\
$(II)$ $\lim\limits_{a\rightarrow0}ak=0$;\\
$(III)$ $\lim\limits_{a\rightarrow0}\frac{1}{a(s_1+pa)^k}=0$;\\
$(IV)$ $\lim\limits_{a\rightarrow0}\frac{1}{a(s_1+pa)^{k_1}}=0$;\\
$(V)$ $\lim\limits_{a\rightarrow0}a^2(s_1+pa)^{k_1}=0$;\\
$(VI)$ $\lim\limits_{a\rightarrow0}W_a^{k_1}(\frac 1 2)=\frac 1 2$.
\end{lemma}
\begin{proof}
Suppose (I) is true. Let us first prove that (II) and (III) are true.

By the definition of $k$, we have:
\begin{equation}\label{L:II1}
0\leq-a^2(s_1+pa)^{k-2}\frac{r(qs_1+ps_2-p-q)+rpqa}{s_1+pa-1}+\frac{s_1-1+pa-2ra}{2(s_1+pa-1)}\leq\frac 12 -\frac{1/2+ra}{s_1+pa}.
\end{equation}
The first inequality of (\ref{L:II1}) implies that $(s_1+pa)^{k-2}\leq\frac{s_1-1+pa-2ra}{2a^2(r(qs_1+ps_2-p-q)+rpqa)}$, thus
$$ak\leq a\frac{\ln(s_1-1+pa-2ra)-\ln 2-2\ln a -\ln(r(qs_1+ps_2-p-q)+rpqa)}{\ln(s_1+pa)}+2a,$$
$$a\leq \frac{\sqrt{s_1-1+pa-2ra}(s_1+pa)}{\sqrt{2(r(qs_1+ps_2-p-q)+rpqa)}(s_1+pa)^{k/2}},$$
$$a^2(s_1+pa)^{k_1}\leq \frac{(s_1-1+pa-2ra)(s_1+pa)^2}{2(r(qs_1+ps_2-p-q)+rpqa)(s_1+pa)^{k-k_1}},$$
so we obtain (V), and since $\lim\limits_{a\rightarrow0}a\ln a=0$, we obtain (II).

The second inequality of (\ref{L:II1}) implies $$\frac{1}{a(s_1+pa)^{k-2}}\leq \frac{2a(r(qs_1+ps_2-p-q)+rpqa)(s_1+pa)}{s_1-1+pa-2ra}.$$ Therefore,
\begin{equation}\label{L:III1}
\frac{1}{a(s_1+pa)^k}\leq \frac{2a(r(qs_1+ps_2-p-q)+rpqa)}{(s_1-1+pa-2ra)(s_1+pa)},
\end{equation}
and as $a\rightarrow0$, we obtain (III).

On the other hand, (\ref{L:III1}) implies
\begin{eqnarray*}
\frac{1}{a(s_1+pa)^{k_1}}&\leq& \frac{2a(r(qs_1+ps_2-p-q)+rpqa)(s_1+pa)^{k-k_1}}{(s_1+pa-2ra-1)(s_1+pa)}\\
&\leq & \frac{\sqrt{2(r(qs_1+ps_2-p-q)+rpqa)}(s_1+pa)^{k-k_1}}{\sqrt{s_1+pa-2ra-1}(s_1+pa)^{k/2}}\\
&=&\frac{\sqrt{2(r(qs_1+ps_2-p-q)+rpqa)}}{\sqrt{s_1+pa-2ra-1}(s_1+pa)^{k_1-k/2}}.
\end{eqnarray*}
By the definition of $k_1$,  we obtain (IV). (VI) follows from (V).

Now, let us prove (I).

The fixed point slightly less than $1/2$ is $x^\ast_l=\frac{s_1-1+pa-2ra}{2(s_1-1+pa)}$, and $$x^\ast_l-W_a^2(1/2)=\frac{ra^2(q(s_1-1)+p(s_2-1)+apq)}{s_1-1+pa}>0,$$
which implies that $W_a^m(1/2)$ are all in the domain of the second branch of $W_a$ for $3\leq m\leq k$. For a linear map $T(x)=m_0x+b_0$, we have $T^n(x)=m_0^nx+\frac{m_0^n-1}{m_0-1}b_0$. This proves (I).
\end{proof}
Using (\ref{simple_fa}) and (\ref{estimate_A}) we see that for the functions
$f_l=1+(1+\frac{s_1+pa}{s_2+qa})\Lambda_lg_h$ and $ f_h=1+(1+\frac{s_1+pa}{s_2+qa})\Lambda_hg_l$, we have
\begin{equation}\label{inequality1}
 f_l\le f_a\le f_h\ .
\end{equation}
Now, we will represent functions $f_l$ and $f_c$ as combinations of functions
$\chi_j$, $j=1,\dots,k_1$ and $\chi_c$. After some calculations, we obtain
\begin{eqnarray*}
f_l
&=& 1+(1+\frac{s_1+pa}{s_2+qa}) \Lambda_l\bigg( \frac{\chi_{[0,W_a(1/2)]}}{s_1+pa}+\frac{1}{s_2+qa}\sum _{j=2}^{k_1}\frac{\chi_{[W_a^j(1/2),1]}}{(s_1+pa)^{j-1}}\\
&&+\frac 1{(s_2+qa)(s_1+pa-1)(s_1+pa)^{k_1-1}}\bigg)\\
&=& \left({\frac{s_1+s_2+pa+qa}{(s_2+qa)(s_1+pa)}}\Lambda_l+1\right)\chi_1+\frac{s_1+s_2+pa+qa}{(s_2+qa)^2}\Lambda_l\sum _{j=2}^{k_1}\frac{\chi_{j}}{(s_1+pa)^{j-1}}\\
&& + \left(\frac{s_1+s_2+pa+qa}{(s_2+pa)^2}\Lambda_l\frac{1-\frac{1}{(s_1+pa)^{k_1-1}}}{s_1+pa-1}+1\right)\chi_c \\
&&+\frac {{\frac{s_1+s_2+pa+qa}{s_2+qa}}\Lambda_l}{(s_2+qa)(s_1+pa-1)(s_1+pa)^{k_1-1}}\ ,
\end{eqnarray*}
\begin{eqnarray*}
f_h
&=& 1+(1+\frac{s_1+pa}{s_2+qa}) \Lambda_h\bigg( \frac{\chi_{[0,W_a(1/2)]}}{s_1+pa}+\frac{1}{s_2+qa}\sum _{j=2}^{k_1}\frac{\chi_{[W_a^j(1/2),1]}}{(s_1+pa)^{j-1}}\bigg)\\
&=& \left({\frac{s_1+s_2+pa+qa}{(s_2+qa)(s_1+pa)}}\Lambda_h+1\right)\chi_1+\frac{s_1+s_2+pa+qa}{(s_2+qa)^2}\Lambda_h\sum _{j=2}^{k_1}\frac{\chi_{j}}{(s_1+pa)^{j-1}}\\
&& + \left(\frac{s_1+s_2+pa+qa}{(s_2+qa)^2}\Lambda_h\frac{1-\frac{1}{(s_1+pa)^{k_1-1}}}{s_1+pa-1}+1\right)\chi_c\ .
\end{eqnarray*}
In the case we are considering, (\ref{estimate_A}) implies that both $\Lambda_l$, $\Lambda_h$ are smaller than -1. Using this,
one can show that all the coefficients in the representation of $f_l$ and $f_h$ are negative
for sufficiently small $a$. For example, let us consider the coefficient of $\chi_1$ in $f_h$:
$${\frac{s_1+s_2+pa+qa}{(s_2+qa)(s_1+pa)}}\Lambda_h+1=\frac{\kappa}{1-(\kappa+\eta)}+1=\frac{1-\eta}{1-(\kappa+\eta)}<0\ .$$

\subsubsection{Normalizations and integrals if $\frac{1}{s_1}+\frac{1}{s_2}=1$ }\label{normequals1}
Let us define $J_1=[0,W_a^{k_1}(1/2)]$,
 $J_2=(W_a^{k_1}(1/2),1/2+ra]$, $J_3=(1/2+ra,1]$.
We will calculate integrals of $f_h$ over each of these intervals $J_1$, $J_2$ and $J_3$, and use them
to normalize $f_h$. We have
\begin{eqnarray*}
C_1& =&\int_{J_1} f_{h} \  d\lambda
=\int_{J_1} \left[{\frac{s_1+s_2+pa+qa}{(s_2+qa)(s_1+pa)}}\Lambda_h+1\right]\chi_1 \  d\lambda\\
 &=&\left[{\frac{s_1+s_2+pa+qa}{(s_2+qa)(s_1+pa)}}\Lambda_h+1\right]W_a^{k_1}(\frac 1 2)
=\left[\frac{\kappa}{1-(\kappa+\eta)}+1\right]W_a^{k_1}(\frac 1 2)\\
 &=&\bigg[\frac{a(2qs_1s_2+ps_2^2-2qs_2-p-q)}{(1-(\kappa+\eta))(s_2+qa)^2(s_1+pa-1)}\\
 &&\quad+\frac{a^2(2pqs_2-q^2+q^2s_1)+pq^2a^3}{(1-(\kappa+\eta))(s_2+qa)^2(s_1+pa-1)}\bigg]W_a^{k_1}(\frac 1 2)\ .
\end{eqnarray*}
Using  Lemma \ref{L:estimates1}, we obtain
$$\lim\limits_{a\rightarrow0}\frac{C_1}{a}=-\frac{2qs_1s_2+ps_2^2-2qs_2-p-q}{2s_2^2(s_1-1)}=-\frac{2qs_1+ps_2^2-p-q}{2s_2s_1}.$$
In the same way, we can see that for any $0<\theta<1/2$, we obtain
$$\lim\limits_{a\rightarrow0}\frac 1 a {\int_0^\theta}f_h d\lambda=-\frac{2qs_1+ps_2^2-p-q}{s_2s_1} \theta\ .$$
On the interval $J_2$, the integral of $f_h$ is:
\begin{eqnarray*}
C_2=\int_{J_2} f_{h} \  d\lambda&=& \int_{J_2} \left[{\frac{s_1+s_2+pa+qa}{(s_2+qa)(s_1+pa)}}\Lambda_h+1\right]\chi_1 \  d\lambda\\
&& + \frac{s_1+s_2+pa+qa}{(s_2+qa)^2}\Lambda_h\sum _{j=2}^{k_1}\int_{J_2}\frac{\chi_{j}}{(s_1+a)^{j-1}} \  d\lambda\\
&=& \frac{1-\eta}{1-(\kappa+\eta)}\left(\frac 1 2+ra-W_a^{k_1}(\frac 1 2)\right)\\
&& + \frac{s_1+s_2+pa+qa}{(s_2+qa)^2}\Lambda_h\bigg[\frac{ra(s_2+qa)}{s_1+pa}+\frac{ra(1-\frac{1}{(s_1+pa)^{k_1-2}})}{(s_1+pa-1)^2}\\
&& +\frac{a^2(k_1-2)}{s_1+pa}\frac{r(qs_1+ps_2-p-q)+rpqa}{s_1+pa-1}\bigg]\ .
\end{eqnarray*}
Using  Lemma \ref{L:estimates1}, we obtain
\[\lim\limits_{a\rightarrow0}\frac{C_2}{a}=-\frac{s_1+s_2}{s_2^2}\left[\frac{rs_2}{s_1}+\frac{r}{(s_1-1)^2}\right]=-rs_2.\]

On the interval $J_3$, the integral of $f_h$ is:
\begin{eqnarray*}
C_3=\int_{J_3} f_{h} \  d\lambda&=& \int_{J_3} \left(\frac{s_1+s_2+pa+qa}{(s_2+qa)^2}\Lambda_h\frac{1-\frac{1}{(s_1+pa)^{k_1-1}}}{s_1+pa-1}+1\right)\chi_c \  d\lambda\\
&=& \left[\left(1-\frac{1}{(s_1+pa)^{k_1-1}}\right)\frac{\eta}{1-(\kappa+\eta)}+1\right](\frac 12-ra)\\
&=& \frac{\frac{a(qs_1+ps_2-p-q)+pqa^2}{(s_1+pa)(s_2+qa)}-\frac{\eta}{(s_1+pa)^{k_1-1}}}{1-(\kappa+\eta)}(\frac{1}{2}-ra)\ .
\end{eqnarray*}
Using Lemma  \ref{L:estimates1}, we obtain
$$\lim\limits_{a\rightarrow0}\frac{C_3}{a}=-\frac{qs_1+ps_2-p-q}{2s_1s_2} \ .$$
In the same way, we can see that for any $0<\theta<1/2$, we obtain
$$\lim\limits_{a\rightarrow0}\frac 1 a {\int_{1/2+\theta}^1}f_h d\lambda=-\frac{qs_1+ps_2-p-q}{s_1s_2} \left(\frac 12-\theta\right)\ .$$

If we define $B=C_1+C_2+C_3$, then $\frac{f_h} B$ is a normalized density. We see that
$$\lim\limits_{a\rightarrow0}\frac{B}{a}=-\frac{(qs_1+ps_2-p-q)(s_2+2)+2rs_1s_2^2}{2s_1s_2} \ .$$

Our calculations show that the normalized measures $\{(f_h/B)\cdot\lambda\}$ converge $*$-weakly
 to the measure $$\frac{(qs_1+ps_2-p-q)(s_2+2)}{(qs_1+ps_2-p-q)(s_2+2)+2rs_1s_2^2}\mu_0+\frac{2rs_1s_2^2}{(qs_1+ps_2-p-q)(s_2+2)+2rs_1s_2^2} \delta_{(\frac 12)}\ .$$

Now, we will show the same holds for the normalized measure defined by $f_l$. To this end, let us notice that
 \begin{eqnarray*}
f_h-f_l&=&(1+\frac{s_1+pa}{s_2+qa})\Lambda_hg_l-(1+\frac{s_1+pa}{s_2+qa})\Lambda_lg_h\\
&=& (1+\frac{s_1+pa}{s_2+qa})(\Lambda_h-\Lambda_l)g_l-\Lambda_l \frac {1+\frac{s_1+pa}{s_2+qa}}{(s_2+qa)(s_1+pa-1)(s_1+pa)^{k_1-1}}\\
&=&
(1+\frac{s_1+pa}{s_2+qa})\frac{\frac{\eta}{(s_1+pa)^{k_1-1}}}{[1-(\kappa+\eta)][1-\kappa-\eta(1-\frac{1}{(s_1+pa)^{k_1-1}})]}g_l\\
& & -\Lambda_l \frac {1+\frac{s_1+pa}{s_2+qa}}{(s_2+qa)(s_1+pa-1)(s_1+pa)^{k_1-1}}\ ,
\end{eqnarray*}
where $|g_l|\le \frac{2}{s_1}$ and $\lim\limits_{a\to 0} \Lambda_l=-1$.
Using Lemma  \ref{L:estimates1} once again,  we can show that for any subinterval $J\subset[0,1]$, we have
\[ \lim_{a\to 0}\frac 1 a \int_J (f_h-f_l)d\lambda=0\ .\]
For $J=[0,1]$ this means that the normalizations of $f_l$  and $f_h$ are asymptotically the same.
With this, the limit for  a general $J$ means in particular that the $*$-weak  limit of
normalized measures defined using $f_l$ is the same as for those defined using $f_h$.
In view of inequality (\ref{inequality1}), this proves Theorem \ref{main}(II).

\subsubsection{Estimates on $f_a$ if $\frac{1}{s_1}+\frac{1}{s_2}<1$}\label{less1}

 We have the following lemma:
\begin{lemma}\label{L:estimates}
For the family of $W_a$ maps, if $\frac{1}{s_1}+\frac{1}{s_2}<1$, we have\\
$(I)$ $W_a(1/2)=1/2+ra$, $W_a^2(1/2)=-ra(s_2+qa)+1/2+ra$, and for $3\leq m\leq k$, we have $W_a^m(1/2)=-a(s_1+pa)^{m-2}\frac{r\left[s_1s_2-s_1-s_2+a(qs_1+ps_2-p-q+pqa)\right]}{s_1+pa-1}+\frac{s_1-1+pa-2ra}{2(s_1+pa-1)}$;\\
$(II)$ $\lim\limits_{a\rightarrow0}ak=0$;\\
$(III)$ $\lim\limits_{a\rightarrow0}a(s_1+pa)^{k_1}=0$;\\
$(IV)$ $\lim\limits_{a\rightarrow0}W_a^{k_1}(\frac 1 2)=\frac 1 2$.
\end{lemma}
\begin{proof}
Suppose (I) is true. Let us first prove that (II) and (III) are true.

By the definition of $k$, we have:
\begin{equation}\label{L:II}
\begin{split}
0\leq &-a(s_1+pa)^{k-2}\frac{r\left[s_1s_2-s_1-s_2+a(qs_1+ps_2-p-q+pqa)\right]}{s_1+pa-1}\\
&+\frac{s_1-1+pa-2ra}{2(s_1+pa-1)}.
\end{split}
\end{equation}
The inequality (\ref{L:II}) implies $a(s_1+pa)^{k-2}\leq\frac{s_1-1+pa-2ra}{2r\left[s_1s_2-s_1-s_2+a(qs_1+ps_2-p-q+pqa)\right]}$, thus
\begin{eqnarray*}
ak\leq &&a\frac{\ln(s_1-1+pa-2ra)-\ln 2+2\ln(s_1+pa)-\ln r-\ln a}{\ln(s_1+pa)}\\
&&-a\frac{\ln(2r\left[s_1s_2-s_1-s_2+a(qs_1+ps_2-p-q+pqa)\right])}{\ln(s_1+pa)},\\
a(s_1+pa)^{k_1}\leq && \frac{(s_1-1+pa-2ra)(s_1+pa)^2}{2r\left[s_1s_2-s_1-s_2+a(qs_1+ps_2-p-q+pqa)\right](s_1+pa)^{k-k_1}},
\end{eqnarray*}
and since $\lim\limits_{a\rightarrow0}a\ln a=0$, we obtain (II) and (III). (IV) follows from (III).

Now, let us prove (I).

The fixed point slightly less than $1/2$ is $x^\ast_l=\frac{s_1-1+pa-2ra}{2(s_1-1+pa)}$, and $$x^\ast_l-W_a^2(1/2)=\frac{ra\left[s_1s_2-s_1-s_2+a(qs_1+ps_2-p-q+pqa)\right]}{s_1-1+pa}>0,$$
which implies that $W_a^m(1/2)$ are all in the domain of the second branch of $W_a$ for $3\leq m\leq k$. Now, (I) follows by the same reasoning as in Lemma \ref{L:estimates1}.
\end{proof}
\begin{lemma}\label{L:convergence}
If the normalized densities $\{h_a\}_{a<a_0}$, for some $a_0>0$,  are uniformly bounded, then $h_a\rightarrow h_0$ in $L^1$.
\end{lemma}
\begin{proof}
The uniform boundedness implies $\{h_a\}_{a<a_0}$ is a weakly precompact set in $L^1$. Thus, any limit of $\{h_a\}_{a<a_0}$ is a invariant density by Proposition 11.3.1 \cite{BG}. At the same time, this limit is an $L^1$ function, thus defines an absolutely continuous invariant measure. Since the map $W_0$ is exact and has only one acim, we conclude that $h_a\rightarrow h_0$ in $L^1$.
\end{proof}
Now, we will prove Theorem \ref{way2}:

The main idea of the proof is the following: since non-normalized densities $\{f_a\}$ are uniformly bounded (formulas (\ref{inequality}, \ref{inequality11}, \ref{L:upper})), it is enough to show that $\{\int^1_0 f_a\ d\lambda\}$ are uniformly separated from zero.

For small $a$, by Lemma \ref{density_fa}, $\Lambda$ (and then both $\Lambda_l$ and $\Lambda_h$) can be either positive or negative. Thus, we can have the following cases.

\textbf{Case (i): $\Lambda_l<0$}:

Comparing with (\ref{simple_fa}) and (\ref{estimate_A}), we see that for the functions
$\widehat{f}_l=1+(1+\frac{s_1+pa}{s_2+qa})\Lambda_lg_h$ and $ \widehat{f}_h=1+(1+\frac{s_1+pa}{s_2+qa})\Lambda_hg_l$, we have
\begin{equation}\label{inequality}
 \widehat{f}_l\le f_a\le \widehat{f}_h\ .
\end{equation}

Note that $\widehat{f}_l$ and $\widehat{f}_h$ have the same form as $f_l$ and $f_h$ in Section \ref{equals1}, so their representations as combinations of functions
$\chi_j$, $j=1,\dots,k_1$ and $\chi_c$ are similar to that of $f_l$ and $f_h$. At the same time, now we have $\frac{1}{s_1}+\frac{1}{s_2}<1$, so the representation is as follows:
\begin{eqnarray*}
\widehat{f}_l
&=& \left({\frac{s_1+s_2+pa+qa}{(s_2+qa)(s_1+pa)}}\Lambda_l+1\right)\chi_1+\frac{s_1+s_2+pa+qa}{(s_2+qa)^2}\Lambda_l\sum _{j=2}^{k_1}\frac{\chi_{j}}{(s_1+pa)^{j-1}}\\
&& + \left(\frac{s_1+s_2+pa+qa}{(s_2+pa)^2}\Lambda_l\frac{1-\frac{1}{(s_1+pa)^{k_1-1}}}{s_1+pa-1}+1\right)\chi_c \\
&&+\frac {{\frac{s_1+s_2+pa+qa}{s_2+qa}}\Lambda_l}{(s_2+qa)(s_1+pa-1)(s_1+pa)^{k_1-1}}\ ,
\end{eqnarray*}
\begin{eqnarray*}
\widehat{f}_h
&=& \left({\frac{s_1+s_2+pa+qa}{(s_2+qa)(s_1+pa)}}\Lambda_h+1\right)\chi_1+\frac{s_1+s_2+pa+qa}{(s_2+qa)^2}\Lambda_h\sum _{j=2}^{k_1}\frac{\chi_{j}}{(s_1+pa)^{j-1}}\\
&& + \left(\frac{s_1+s_2+pa+qa}{(s_2+qa)^2}\Lambda_h\frac{1-\frac{1}{(s_1+pa)^{k_1-1}}}{s_1+pa-1}+1\right)\chi_c\ .
\end{eqnarray*}
(\ref{estimate_A}) implies that all the coefficients in the representation of $\widehat{f}_l$ and $\widehat{f}_h$ are negative
for sufficiently small $a$.

We use the same notations $J_1$, $J_2$ and $J_3$ as in Section \ref{normequals1}. First, we do the calculations assuming that $\vartheta=1-\left(\frac{s_1+s_2}{s_1s_2}+\frac{s_1+s_2}{s_2^2(s_1-1)}\right)\neq 0$.

We will calculate the integrals of $\widehat{f}_h$ over each of $J_1$, $J_2$ and $J_3$, and use them
to normalize $\widehat{f}_h$. We have
\begin{eqnarray*}
\widehat{C}_1&=&\int_{J_1} \widehat{f}_{h} \  d\lambda
= \int_{J_1} \left[{\frac{s_1+s_2+pa+qa}{(s_2+qa)(s_1+pa)}}\Lambda_h+1\right]\chi_1 \  d\lambda\\
&=& \left[{\frac{s_1+s_2+pa+qa}{(s_2+qa)(s_1+pa)}}\Lambda_h+1\right]W_a^{k_1}(\frac 1 2)
= \left[\frac{\kappa}{1-(\kappa+\eta)}+1\right]W_a^{k_1}(\frac 1 2)\\
&=& \bigg[\frac{s_1s_2^2-s_1-s_2-s_2^2}{1-(\kappa+\eta))(s_2+qa)^2(s_1+pa-1)}\\
&& +\frac{a(2qs_1s_2+ps_2^2-2qs_2-p-q)}{(1-(\kappa+\eta))(s_2+qa)^2(s_1+pa-1)}\\
&& +\frac{a^2(2pqs_2-q^2+q^2s_1)+pq^2a^3}{(1-(\kappa+\eta))(s_2+qa)^2(s_1+pa-1)}\bigg]W_a^{k_1}(\frac 1 2)\ .
\end{eqnarray*}
Using  Lemma \ref{L:estimates}, we have
\begin{eqnarray*}
\lim\limits_{a\rightarrow0}\widehat{C}_1=\frac{1}{2}\frac{\frac{s_1s_2^2-s_1-s_2-s_2^2}{s_2^2(s_1-1)}}{1-\left(\frac{s_1+s_2}{s_1s_2}+\frac{s_1+s_2}{s_2^2(s_1-1)}\right)}
=\frac{1}{2}\frac{1-\frac{s_1+s_2}{s_2^2(s_1-1)}}{1-\left(\frac{s_1+s_2}{s_1s_2}+\frac{s_1+s_2}{s_2^2(s_1-1)}\right)}.
\end{eqnarray*}

On the interval $J_2$, the integral of $\widehat{f}_h$ is:
\begin{eqnarray*}
\widehat{C}_2=\int_{J_2} \widehat{f}_{h} \  d\lambda&=& \int_{J_2} \left[{\frac{s_1+s_2+pa+qa}{(s_2+qa)(s_1+pa)}}\Lambda_h+1\right]\chi_1 \  d\lambda\\
&& + \frac{s_1+s_2+pa+qa}{(s_2+qa)^2}\Lambda_h\sum _{j=2}^{k_1}\int_{J_2}\frac{\chi_{j}}{(s_1+pa)^{j-1}} \  d\lambda\\
&=& \frac{1-\eta}{1-(\kappa+\eta)}\left(\frac 1 2+ra-W_a^{k_1}(\frac 1 2)\right)\\
&& + \frac{s_1+s_2+pa+qa}{(s_2+qa)^2}\Lambda_h\bigg[\frac{ra(s_2+qa)}{s_1+pa}+\frac{ra(1-\frac{1}{(s_1+pa)^{k_1-2}})}{(s_1+pa-1)^2}\\
&& +\frac{a(k_1-2)}{s_1+pa}\frac{r(s_1s_2-s_1-s_2+a(qs_1+ps_2-p-q+pqa))}{s_1+pa-1}\bigg]\ .
\end{eqnarray*}
Using  Lemma \ref{L:estimates}, we have
$\lim\limits_{a\rightarrow0} \widehat{C}_2=0.$

On the interval $J_3$, the integral of $\widehat{f}_h$ is:
\begin{eqnarray*}
\widehat{C}_3=\int_{J_3} \widehat{f}_{h} \  d\lambda&=& \int_{J_3} \left(\frac{s_1+s_2+pa+qa}{(s_2+qa)^2}\Lambda_h\frac{1-\frac{1}{(s_1+pa)^{k_1-1}}}{s_1+pa-1}+1\right)\chi_c \  d\lambda\\
&=& \left[\left(1-\frac{1}{(s_1+pa)^{k_1-1}}\right)\frac{\eta}{1-(\kappa+\eta)}+1\right](\frac 12-ra)\\
&=& \frac{\frac{s_1s_2-s_1-s_2+a(qs_1+ps_2-p-q)+pqa^2}{(s_1+pa)(s_2+qa)}-\frac{\eta}{(s_1+pa)^{k_1-1}}}{1-(\kappa+\eta)}(\frac{1}{2}-ra)\ .
\end{eqnarray*}
Using Lemma  \ref{L:estimates} once again, we have
$$\lim\limits_{a\rightarrow0}\widehat{C}_3=\frac{1}{2}\frac{1-\frac{s_1+s_2}{s_1s_2}}{1-\left(\frac{s_1+s_2}{s_1s_2}+\frac{s_1+s_2}{s_2^2(s_1-1)}\right)} \ .$$

Note that if we define $\widehat{B}=\widehat{C}_1+\widehat{C}_2+\widehat{C}_3$, then
$$\lim\limits_{a\rightarrow0}\widehat{B}=\frac{1}{2}\frac{2-\left(\frac{s_1+s_2}{s_1s_2}+\frac{s_1+s_2}{s_2^2(s_1-1)}\right)}{1-\left(\frac{s_1+s_2}{s_1s_2}+\frac{s_1+s_2}{s_2^2(s_1-1)}\right)} \ ,$$
which is not 0. Since $\{\widehat{f}_{h}\}$ are uniformly bounded, we conclude that the normalized $\{\widehat{f}_{h}\}$ are also uniformly bounded.

Now, we will show that the normalized $\{\widehat{f}_l\}$ are also uniformly bounded. To this end, let us notice that
 \begin{eqnarray*}
\widehat{f}_h-\widehat{f}_l&=&(1+\frac{s_1+pa}{s_2+qa})\Lambda_hg_l-(1+\frac{s_1+pa}{s_2+qa})\Lambda_lg_h\\
&=& (1+\frac{s_1+pa}{s_2+qa})(\Lambda_h-\Lambda_l)g_l-\Lambda_l \frac {1+\frac{s_1+pa}{s_2+qa}}{(s_2+qa)(s_1+pa-1)(s_1+pa)^{k_1-1}}\\
&=&
(1+\frac{s_1+pa}{s_2+qa})\frac{\frac{\eta}{(s_1+pa)^{k_1-1}}}{[1-(\kappa+\eta)][1-\kappa-\eta(1-\frac{1}{(s_1+pa)^{k_1-1}})]}g_l\\
& & -\Lambda_l \frac {1+\frac{s_1+pa}{s_2+qa}}{(s_2+qa)(s_1+pa-1)(s_1+pa)^{k_1-1}}\ ,
\end{eqnarray*}
where $|g_l|\le \frac{1}{s_1}+\frac{1}{s_2(s_1-1)}$ and $\lim\limits_{a\to 0} \Lambda_l=\frac{1}{1-\left(\frac{s_1+s_2}{s_1s_2}+\frac{s_1+s_2}{s_2^2(s_1-1)}\right)}$.
Thus,
$ \lim\limits_{a\to 0} \widehat{f}_h-\widehat{f}_l=0\ .$
We conclude that the normalized $\{\widehat{f}_{l}\}$ are uniformly bounded since the normalized $\{\widehat{f}_{h}\}$ are uniformly bounded. Thus, after normalization, $\{f_a\}$ are also uniformly bounded.

\textbf{Case (ii): $\Lambda_l>0$}:

This case implies that $f_a$ given by (\ref{simple_fa}) has the following properties:
\begin{equation}\label{inequality11}
  f_a\geq 1\ ,
\end{equation}
and all the coefficients of the characteristic functions appearing in (\ref{simple_fa}) are positive. We note that $\Lambda$ is always positive for small $a$. Thus,
\begin{equation}\label{L:upper}
 f_a\leq 1+ (1+\frac{s_1+pa}{s_2+qa})\Lambda\sum _{n=1}^\infty\frac{1}{|\beta(1/2,n)|}\ ,
\end{equation}
which is finite since our maps $\{W_a\}$ are expanding. In view of (\ref{inequality11}), we conclude that the normalized $\{f_a\}$ are uniformly bounded.

If $\vartheta=1-\left(\frac{s_1+s_2}{s_1s_2}+\frac{s_1+s_2}{s_2^2(s_1-1)}\right)=0$, then we have $\lim\limits_{a\rightarrow0}\frac{1}{\Lambda_l}=\lim\limits_{a\rightarrow0}\frac{1}{\Lambda_h}=0$, $\Lambda_l$ and $\Lambda_h$ are still of the same sign. We can renormalize $f_a$. Let us take the $\widehat{f}_h$ as an example. Multiplying it by $\frac{1}{\Lambda_h}$, we obtain
\begin{eqnarray*}
\frac{1}{\Lambda_h}\widehat{f}_h
&=& \left({\frac{s_1+s_2+pa+qa}{(s_2+qa)(s_1+pa)}}+\frac{1}{\Lambda_h}\right)\chi_1+\frac{s_1+s_2+pa+qa}{(s_2+qa)^2}\sum _{j=2}^{k_1}\frac{\chi_{j}}{(s_1+pa)^{j-1}}\\
&& + \left(\frac{s_1+s_2+pa+qa}{(s_2+qa)^2}\frac{1-\frac{1}{(s_1+pa)^{k_1-1}}}{s_1+pa-1}+\frac{1}{\Lambda_h}\right)\chi_c\ .
\end{eqnarray*}
Note that the coefficients of $\chi_1$ and $\chi_c$ converge to $\frac{s_1+s_2}{s_1s_2}$ and $\frac{s_1+s_2}{s_2^2(s_1-1)}$, respectively. Thus, $\{\int_0^1\frac{1}{\Lambda_h}\widehat{f}_h\ d\lambda\}$ are separated from 0. This implies $\{\frac{1}{\Lambda_h}\widehat{f}_h\}$ are uniformly bounded. A similar procedure can be applied to $\widehat{f}_l$. We conclude that $\{\frac{1}{\Lambda}f_a\}$ are uniformly bounded.

\section{Example} \label{S:example}
One of the important properties of a piecewise expanding transformation of an interval is that its invariant density is bounded away from 0 on its support. The following result was proved, by Keller \cite{K78} and by Kowalski \cite{Ko}.
\begin{theorem}\label{kk}
Let a transformation $\tau: I\rightarrow I$ be piecewise expanding with $\frac{1}{|\tau'(x)|}$ a function of bounded variation, and let $f$ be a $\tau$-invariant density which can be assumed to be lower semicontinuous. Then there exists a constant $c>0$ such that $f|_{\emph{supp}\ f}>c$.
\end{theorem}
We provide an example showing that this result cannot be generalized to a family of expanding maps, even if they all have this property and converge to a limit map also with this property. Let $d(\cdot,\cdot)$ be the metric on the weak topology of measures.
\begin{example}
Let us fix
$$s_1=s_2=2, \ p=q=1.$$
For small $a>0$, let $W_{a,r}$ denote the $W_a$ maps with varying parameter $r$, and let $\mu_{a,r}$ denote the absolutely continuous invariant measure of $W_{a,r}$. We know that $\mu_{a,r}$ is supported on $[0,1]$ and $W_{a,r}$ with $\mu_{a,r}$ is exact. Using Theorem \ref{main}, we know that  $\{\mu_{a,r}\}$ converge $*$-weakly to the
 measure
 $$\mu_{0,r}=\frac{1}{1+2r}\mu_{0}+\frac{2r}{1+2r}\delta_{\frac 12}.$$
Let $r_n=n$, $n=1,2,3,\cdots$. Also, let $\{a_n\}_1^{\infty}$ satisfy $r_na_n<1/2$ and be so small that
$$d(\mu_{a_n,r_n},\mu_{0,r_n})<\frac 1n\ .$$

Now, for the family of maps $\tau_n=W_{a_n,r_n}$, $n=1,2,3,\cdots$, $\tau_n$  converge to $W_0$ with $|\tau'_n(x)|>2$, but the invariant densities $\mu_{a_n,r_n}$ converge to $\delta_{(\frac 12)}$. This implies that the invariant densities $\{f_{a_n,r_n}\}$ corresponding to $\{\mu_{a_n,r_n}\}$ have no uniform positive lower bound.
\end{example}
\bigskip
\textbf{Acknowledgment:} The author is grateful to Dr. P. G\'ora and Dr. A. Boyarsky, for  inspiring the author and for help with this paper. He would also like to thank the members of the dynamical system seminars at Concordia University for helpful discussions.

\end{document}